\documentclass{amsart}
\usepackage{amsmath, amscd, amssymb, amsthm}
\usepackage{bbm}
\usepackage{latexsym}
\usepackage{amsfonts}
\usepackage{graphicx}
\usepackage[all,cmtip]{xy}
\usepackage[colorlinks,linkcolor=blue,breaklinks=blue,urlcolor=blue,citecolor=blue,anchorcolor=blue,pagebackref]{hyperref}%
\usepackage{geometry}
\newtheorem{theorem}{Theorem}
\newtheorem{lemma}{Lemma}

\newtheorem{remark}{Remark}
\newtheorem{proposition}{Proposition}

\newtheorem{conjecture}{Conjecture}

\renewcommand*\backref[1]{}
\renewcommand*\backrefalt[4]{ \ifcase #1 \or (cited on page #2) \else (cited on pages #2) \fi}

\newcommand{\be}{\begin{equation}}
\newcommand{\ee}{\end{equation}}
\newcommand{\bea}{\begin{eqnarray}}
\newcommand{\eea}{\end{eqnarray}}

\newcommand{\vs}{\vspace{0.5cm}}

\newcommand{\vsv}{\vspace{0.12cm}}

\def\XXint#1#2#3{{\setbox0=\hbox{$#1{#2#3}{\int}$ }
\vcenter{\hbox{$#2#3$ }}\kern-.6\wd0}}

\begin{document}

\title[Strominger connection and pluriclosed metrics]{Strominger connection and pluriclosed metrics}
	
\author{Quanting Zhao}
\address{Quanting Zhao. School of Mathematics and Statistics \&
Hubei Key Laboratory of Mathematical Sciences, Central China Normal
University, Wuhan, 430079, P.R.China.} \email{zhaoquanting@126.com;zhaoquanting@mail.ccnu.edu.cn}
\thanks{Zhao is partially supported by National Natural Science Foundations of China with the grant No.11801205.
Zheng is partially supported by National Natural Science Foundations of China
with the grant No.12071050 and 12141101, Chongqing grant cstc2021ycjh-bgzxm0139, and is supported by the 111 Project D21024.}

\author{Fangyang Zheng}
\address{Fangyang Zheng. School of Mathematical Sciences, Chongqing Normal University, Chongqing 401331, China}
\email{20190045@cqnu.edu.cn;franciszheng@yahoo.com} \thanks{}


\begin{abstract}
In this paper, we prove a conjecture raised by Angella, Otal, Ugarte, and Villacampa recently, which states that if the Strominger connection (also known as Bismut connection) of a compact Hermitian manifold is K\"ahler-like, in the sense that its curvature tensor obeys all the symmetries of the curvature of a K\"ahler manifold, then the metric must be pluriclosed. What we actually showed is a bit more: for any given Hermitian manifold, the Strominger K\"ahler-like condition  is equivalent to the pluriclosedness of the metric plus the parallelness of the torsion.
\end{abstract}

\maketitle

\tableofcontents

\markleft{Quanting Zhao and Fangyang Zheng}
\markright{Strominger K\"ahler-like}

\section{Introduction}

Given a Hermitian manifold $(M^n,g)$, the {\em Strominger connection} $\nabla^s$ (also known as the {\em Bismut connection}) is the unique connection on $M$ that is Hermitian (namely, $\nabla^sg=0$, $\nabla^s J =0$) and has totally skew-symmetric torsion tensor. Its explicit expression appeared in Strominger's paper \cite{Strominger} in 1986, where he called it the H-connection, and independently in Bismut's paper \cite{Bismut} where he established the existence and uniqueness as well as using it in his study of local index theorems. We note that the connection has appeared implicitly earlier in math literature (see for instance \cite{Yano}) and was also used by physicists earlier (see for example \cite{CHSW}, \cite{GatesHR} and \cite{Hull}). So one could argue for calling it {\em Bismut-Hull-Strominger connection,} which is too long. In some literature it was also called the {\em KT connection} (K\"ahler with torsion). Since the need of non-K\"ahler Calabi-Yau spaces in string theory, more specifically in Hull-Strominger system in type II string theory and in 2-dimensional supersymmetric
$\sigma$-models, this particular connection has been drawing more and more attention from geometers and mathematical physicists alike. We refer the readers to the papers \cite{FY}, \cite{Fu}, \cite{Fu-Li-Yau}, \cite{Fu-Yau}, \cite{Li-Yau}, \cite{Tseng-Yau} for study involving Hull-Strominger system, which uses Bismut-Strominger connection in an essential way, and to the papers \cite{S18},  \cite{ST10}, \cite{ST13} where the study of Hermitian curvature flow involving the connection were initiated which had many follow-ups in recent years.


A Hermitian metric $g$ is called {\em pluriclosed} if $\partial \overline{\partial }\omega =0$, where $\omega$ is the K\"ahler form of $g$. This type of metrics is also called {\em Strong KT metric} (or {\em SKT metric}) in many literature (see for example the nice survey paper by Fino and Tomassini \cite{FinoTomassini}). It is an obvious generalization of the K\"ahlerness condition (which is $d\omega =0$). The pluriclosed metrics have been studied by many authors. For this and related topics in non-K\"ahler geometry such as Calabi-Yau problem, vanishing theorems, Gauduchon connections, balanced metrics, K\"ahler-like conditions, etc., we refer the readers to  \cite{AI}, \cite{EFV}, \cite{FV},    \cite{Fu-Zhou},   \cite{Gauduchon1}, \cite{IvanovP}, \cite{KYZ},  \cite{Liu-Yang}, \cite{Liu-Yang1}, \cite{Liu-Yang2},    \cite{STW},  \cite{Tosatti},   \cite{VYZ}, \cite{YZ1}, \cite{Zheng1} and the references therein for more discussions and backgrounds.


For a connection $D$ on $(M^n,g)$, its curvature tensor $R^D$ is given by
\begin{equation} \label{eq:def-curvature}
 R^D(X,Y,Z,W) = \langle D_XD_YZ - D_YD_XZ- D_{[X,Y]}Z, \, W \rangle ,
 \end{equation}
where $g(\, , \, ) = \langle \, , \, \rangle$ and $X$, $Y$, $Z$, $W$ are tangent vectors in $M^n$. $R^D$ is  skew-symmetric with respect to its first two positions by definition, and it will be skew-symmetric with respect to its last two positions if $Dg=0$,  namely if $D$ is a metric connection.


A metric connection $D$ on $(M^n,g)$ is called {\em K\"ahler-like} if its curvature $R^D$ satisfies the symmetry conditions
\begin{eqnarray*}
&& R^D(X,Y,Z,W) + R^D(Y,Z,X, W) + R^D(Z, X,Y,W) = 0, \\
&& R^D(X,Y,JZ,JW)= R^D(X,Y,Z,W) = R^D(JX,JY,Z,W)
\end{eqnarray*}
for any tangent vectors $X$, $Y$, $Z$, $W$ in $M^n$. The first line is the {\em first Bianchi identity}, and the second line is called the {\em type condition}. Note that the first equality in the type condition is always satisfied when $DJ=0$.

If we extend $R^D$ linearly over ${\mathbb C}$, and use the decomposition $TM\otimes {\mathbb C} = T^{1,0}M \oplus T^{0,1}M$, where $T^{1,0}M$ consists of all vector fields of the form $x= X-\sqrt{-1}JX$ where $X$ is real, then the above definition is equivalent to
\begin{eqnarray*}
&& R^D(x,\overline{y}, z, \overline{w}) = R^D(z,\overline{y}, x, \overline{w}), \\
&& R^D(x, y, \ast , \ast )= R^D(\ast , \ast , z, w) = 0
\end{eqnarray*}
for any type $(1,0)$ complex tangent vectors $x$, $y$, $z$, and $w$. In other words, in complex terms, the only possibly non-trivial components of $R^D$ are $R^D(x, \overline{y}, z, \overline{w})$,  and $R^D$ is symmetric when the first and the third positions are swapped\footnote{Note that this swap is different than requiring $R^D(X,Y,Z,W)=R^D(Z,Y,X,W)$ for any real vectors, which is a strictly stronger condition. In an earlier version of our manuscript, we mistakenly thought that they are the same and stated the definition of K\"ahler-likeness incorrectly (even though we used the correct definition in mind so none of the later discussions were affected). This was kindly pointed out to us by Fino and Tardini. See \cite{FinoTardini} for more details.}.

This notion was introduced in \cite{YZ} in 2018 for Levi-Civita (namely, Riemannian) and Chern connections, following the pioneer work of Gray \cite{Gray} and others. In \cite{AOUV}, Angella, Otal, Ugarte and Villacampa generalized it to any metric connection, and they particularly studied it for the Strominger connection $\nabla^s$ and the one-parameter family of canonical connections called the {\em Gauduchon connections} which we will  denote as $\nabla^{(t)} = (1-\frac{t}{2})\nabla^c + \frac{t}{2}\nabla^s$, where $t\in {\mathbb R}$ and $\nabla^c$ is the Chern connection.


Through a detailed study on all nilmanifolds and Calabi-Yau type solvmanifolds of dimension three, they classified all those spaces which are $\nabla^{(t)}$-K\"ahler-like, and they proposed the following conjecture:

\begin{conjecture}[AOUV\cite{AOUV}]
For a compact Hermitian manifold $(M^n,g)$, if the Strominger connection is K\"ahler-like, then $g$ must be pluriclosed.
\end{conjecture}

In \cite{AOUV}, the authors proved the above conjecture under the assumption that $\nabla^s$ is flat (i.e., the curvature of $\nabla^s$ vanishes), using the classification result of \cite{WYZ} which says that all Strominger flat (which was called Bismut flat in that paper) manifolds are covered by Samelson spaces \cite{S}.

As proved in \cite{AOUV}, there are examples of compact Hermitian manifolds which are Strominger K\"ahler-like, but not Strominger flat. The simplest such example is a primary Kodaira surface.

The main purpose of the present paper is to give an affirmative answer to the above conjecture, and it turns out that the result is actually true even without the compactness assumption. That is, we have the following

\begin{theorem}\label{thm1}
Let $(M^n,g)$ be a Hermitian manifold. Its Strominger connection $\nabla^s$ is K\"ahler-like if and only if $\nabla^s$ has parallel torsion and $g$ is pluriclosed.
\end{theorem}

The main technical part of the proof is to show that if the Strominger connection $\nabla^s$ is K\"ahler-like, then its torsion is parallel. When the Strominger connection is flat, the authors of \cite{WYZ} used a Bochner identity trick analogous to the famous work of Boothby \cite{Boothby} who classified all compact Chern flat manifolds, to conclude that the Chern torsion $T^c$ is $\nabla^s$-parallel (see also \cite{Wang} for a more general discussion of complex parallelizable manifolds).


When $\nabla^s$ is K\"ahler-like but not flat, this argument no longer works, and one has to rely on the deep algebraic tanglement of the torsion and its covariant derivatives to show its parallelness. The strategy is to walk in two steps, first to focus on the Gauduchon torsion $1$-form $\eta$, which is the trace of the full torsion tensor $T^c$, and show that $\eta$ is $\nabla^s$-parallel. Then in the second step, we use the parallelness of $\eta$ to further analyze $T^c$ and establish its parallelness.


In complex dimension $2$, $T^c$ and $\eta$ carry the same amount of information, and the situation becomes particularly simple. In this case, the Strominger K\"ahler-like condition is actually equivalent to $\nabla^sT^c=0$, and it implies that the metric is pluriclosed. Furthermore, the Strominger K\"ahler-like condition is also equivalent to a known condition called {\em Vaisman}.


Recall that a Hermitian manifold $(M^n,g)$ is said to be {\em locally conformally K\"ahler}, if there exists a closed $1$-form $\psi$ on $M^n$ such that $d\omega =\omega \wedge \psi$. This $\psi$ is called the {\em Lee form} of $(M^n,g)$. The Hermitian manifold is said to be {\em Vaisman,} if it is locally conformally K\"ahler and its Lee form is parallel under the Levi-Civita connection. Note that for a Hermitian surface $(M^2,g)$, one always has $d\omega =-2\omega \wedge (\eta +\overline{\eta})$, where $\eta$ is Gauduchon's torsion $1$-form. So $(M^2,g)$ is Vaisman if and only if $\eta +\overline{\eta}$ is parallel under the Levi-Civita connection.   We prove that for $n=2$, the Strominger K\"ahler-like condition is equivalent to the Vaisman condition:

\begin{theorem}\label{thm2}
Let $(M^2,g)$ be a Hermitian surface. Then the following are equivalent:
\begin{enumerate}
\item The Strominger connection is K\"ahler-like.
\item $\nabla^s T^c=0$.
\item $g$ is Vaisman, namely, its Lee form is parallel under the Levi-Civita connection.
\end{enumerate}
\end{theorem}

Now let $(M^2,g)$ be a compact Hermitian surface which is Strominger K\"ahler-like, or equivalently, Vaisman. When $b_1(M)$ is odd, Belgun  gave in his beautiful work \cite{Belgun} a complete classification of all such metrics. In particular, $M^2$ is either a properly elliptic surface, a Kodaira surface \cite{Kodaira}, or an elliptic or Class 1 Hopf surface. When $b_1(M)$ is even, the surface admits K\"ahler metrics, which will force $g$ to be K\"ahler, so there is no non-K\"ahler Vaisman metric on such surfaces. This phenomenon persists in higher dimensions as well. To state the result, let us first recall the following interesting conjecture of Fino and Vezzoni in non-K\"ahler geometry:

\begin{conjecture}[Fino-Vezzoni \cite{FV2}]
If a compact complex manifold $M^n$ admits a pluriclosed metric $g$ and a balanced metric $h$, then it must admit a K\"ahler metric.
\end{conjecture}

Our main result says that Strominger K\"ahler-like metrics are always pluriclosed. So the following result can be regarded as a partial evidence to the above conjecture:

\begin{theorem}\label{thm3}
Let $(M^n,g)$ be a compact Hermitian manifold that is Strominger K\"ahler-like. If $g$ is not K\"ahler, then $M^n$ does not admit any balanced metric.
\end{theorem}

In other words, given any compact Strominger K\"ahler-like manifold $(M^n,g)$, if $M^n$ admits a balanced metric $g_0$, then $g$ must be K\"ahler.  We will actually prove a slightly stronger statement: if $g$ is not K\"ahler, such $M^n$ does not admit any metric that is {\em strongly Gauduchon} in the sense of Popovici \cite{Popovici}, which is a weaker condition than balanced (see Theorem \ref{thm4}  in \S 5). On the other hand, we will show as Proposition \ref{prop3} in \S 5 that Strominger K\"ahler-like metrics are always {\em Gauduchon,} namely, $\partial \overline{\partial} (\omega^{n-1})=0$. In fact, for any $1\leq k\leq n-1$, they are always {\em $k$-Gauduchon} in the sense of Fu-Wang-Wu \cite{FuWangWu}: $\partial \overline{\partial} (\omega^{k}) \wedge \omega^{n-k-1}=0$. We will also show that for any compact, non-K\"ahler, Strominger K\"ahler-like manifold $M$, the Dolbeault cohomology group $H^{0,1}_{\overline{\partial}}(M)$ must be non-trivial (see Proposition \ref{prop4} in \S 5).

It seems that compact Strominger K\"ahler-like manifolds form a rather restrictive class, and particularly so in low dimensions. After the completion of the present paper, we were able to obtain two results: one is the explicit description of Strominger K\"ahler-like manifolds amongst all complex nilmanifolds with nilpotent complex structures in all dimensions \cite{ZZ1}, the other is the classification of  compact Strominger K\"ahler-like manifolds in complex dimension $3$ \cite{YZZ}.

The paper is organized as follows: In Section \ref{plr}, we  collect some preliminary results and fix the notations. In Section \ref{BKmetric}, we examine the basic properties for Strominger K\"ahler-like metrics, and give proofs to Theorem \ref{thm1} and \ref{thm2} assuming the main technical result Proposition \ref{prop2}. In Section \ref{pll_tor}, we show the parallelness of the torsion tensor and establish the proof of Proposition \ref{prop2}. In the last section,  we will prove Theorem \ref{thm3} and its slight generalization Theorem \ref{thm4} (about the non-existence of strongly Gauduchon metrics), and we also observe some properties for Strominger K\"ahler-like metrics and prove Propositions \ref{prop3} and \ref{prop4}.

\vs

\section{Preliminaries}\label{plr}

In this section, we collect some known results for our later use and also fix the notations. It is included here to make the paper self-contained, for the convenience of the readers, since the proof of the main theorem is computational in nature. We refer the readers to \cite{YZ} and \cite{WYZ} for more details, and to \cite{Zheng} as a more general reference.


Let $(M^n,g)$ be a Hermitian manifold, where $n\geq 2$. We will
denote by $\nabla$, $\nabla^c$, and $\nabla^s$ respectively the Levi-Civita (we will also call it {\em Riemannian} for convenience), Chern, and Strominger (or Bismut) connection of the metric $g$, and by $R$, $R^c$, and $R^s$  their curvatures, called
the Riemannian, Chern, or Strominger curvature tensor, respectively.


We will denote by  $T^{1,0}M$ the bundle of complex tangent vector fields of type $(1,0)$, namely,
complex vector fields of the form  $v- \sqrt{-1}Jv$, where $v$ is a real vector field on $M$. Let $\{ e_1, \ldots , e_n\}$ be a local frame of $T^{1,0}M$ in a neighborhood in $M$. Write $e=\ ^t\!(e_1, \ldots , e_n) $
 as a column vector. Denote by $\varphi = \ ^t\!(\varphi_1, \ldots ,
  \varphi_n)$ the column vector of local $(1,0)$-forms which is the
coframe dual to $e$. For the Chern connection $\nabla^c$ of $g$,
let us denote by $\theta$,  $\Theta$ the matrices of connection and
curvature, respectively, and by $\tau$ the column vector of the
torsion $2$-forms, all under the local frame $e$. Then the structure
equations and Bianchi identities are
\begin{eqnarray*}
d \varphi & = & - \ ^t\!\theta \wedge \varphi + \tau,  \label{formula 1}\\
d  \theta & = & \theta \wedge \theta + \Theta. \\
d \tau & = & - \ ^t\!\theta \wedge \tau + \ ^t\!\Theta \wedge \varphi, \label{formula 3} \\
d  \Theta & = & \theta \wedge \Theta - \Theta \wedge \theta.
\end{eqnarray*}
The entries of $\Theta$ are all $(1,1)$ forms, while the entries of the column vector $\tau $ are all $(2,0)$ forms, under any frame $e$.


Write  $\langle \ , \rangle $ for the (real) inner product given by the Hermitian metric $g$, and extend it bilinearly over ${\mathbb C}$. Under the frame $e$, let us denote the components of the Riemannian connection $\nabla$ by
$$ \nabla e = \theta_1 e + \overline{\theta_2 }\overline{e} ,
\ \ \ \nabla \overline{e} = \theta_2 e + \overline{\theta_1
}\overline{e} ,$$ then the matrices of connection and curvature for
$\nabla $ become:
$$ \hat{\theta } = \left[ \begin{array}{ll} \theta_1 & \overline{\theta_2 } \\ \theta_2 & \overline{\theta_1 }  \end{array} \right] , \ \  \  \hat{\Theta } = \left[ \begin{array}{ll} \Theta_1 & \overline{\Theta}_2  \\ \Theta_2 & \overline{\Theta}_1   \end{array} \right], $$
 where
\begin{eqnarray*}
\Theta_1 & = & d\theta_1 -\theta_1 \wedge \theta_1 -\overline{\theta_2} \wedge \theta_2, \\
\Theta_2 & = & d\theta_2 - \theta_2 \wedge \theta_1 - \overline{\theta_1 } \wedge \theta_2,  \label{formula 7}\\
d\varphi & = & - \ ^t\! \theta_1 \wedge \varphi - \ ^t\! \theta_2
\wedge \overline{\varphi } .
\end{eqnarray*}
Also, for the Strominger connection $\nabla^s$, we will write
\begin{equation*}
\nabla^s e = \theta^s e, \ \ \ \Theta^s=d\,\theta^s - \theta^s \wedge \theta^s,
\end{equation*}
for the matrices of connection and curvature under the frame $e$. When $e$ is unitary,
both $\theta_2 $ and $\Theta_2$ are skew-symmetric, while $\theta$, $\theta_1$, $\theta^s$,
 or $\Theta$,  $\Theta_1$, $\Theta^s$ are all skew-Hermitian.


Following \cite{YZ}, we will introduce a $(2,1)$ tensor $\gamma $ by letting its components under the frame $e$ be the matrix of $1$-forms (which we will denote by the same letter for convenience)
\begin{equation*}
\gamma = \theta_1 - \theta ,
\end{equation*}
and denote by $\gamma = \gamma ' + \gamma ''$  the decomposition
of $\gamma$ into $(1,0)$ and $(0,1)$ parts. By \cite[Lemma 2]{WYZ}, we have
\begin{equation*}
\theta^s = \theta + 2\gamma = \theta_1 +\gamma ,
\end{equation*}
and more generally,  consider the line of canonical connections on $(M^n,g)$: the {\em $t$-Gauduchon connection} $\nabla^{(t)} = (1-\frac{t}{2})\nabla^c + \frac{t}{2}\nabla^s$ where $t\in {\mathbb R}$,  whose matrix of connection under the frame $e$ is given by $\theta^{(t)}=\theta + t \gamma$.

Next let us denote by $T_{ij}^k=-T_{ji}^k$ the components of $\tau$:
\begin{equation*}
\tau_k = \sum_{i,j=1}^n T_{ij}^k \varphi_i\wedge \varphi_j \ = \ 2\!\sum_{1\leq i<j\leq n}  T_{ij}^k \varphi_i\wedge \varphi_j.
\end{equation*}
Note that our $T_{ij}^k$ is only half of the components of the torsion $\tau$ used in some other literature where the second sigma term is used. Also, if we denote by $T^c$ the torsion tensor of the Chern connection, namely,
$$ T^c(X,Y) = \nabla^c_XY - \nabla^c_YX - [X,Y], $$
then we have
$$ T^c(e_i, \overline{e}_j)=0, \ \ \ T^c(e_i, e_j) = 2 \sum_{k=1}^n T_{ij}^k e_k. $$
As observed in \cite{YZ},  when $e$ is unitary, $\gamma $ and $\theta_2$ take the following simple forms:
\begin{equation*}
(\theta_2)_{ij} = \sum_{k=1}^n \overline{T^k_{ij}} \varphi_k, \ \ \ \ \gamma_{ij} = \sum_{k=1}^n ( T_{ik}^j \varphi_k - \overline{T^i_{jk}} \overline{\varphi}_k ).
\end{equation*}
So the torsion tensor $T^{(t)}$ for the $t$-Gauduchon connection has components:
\begin{equation*}
T^{(t)}(e_i, e_j)  = (2-2t) \sum_{k=1}^n T^k_{ij} e_k, \ \ \ \ T^{(t)}(e_i, \overline{e}_j)  = t\sum_{k=1}^n ( \overline{T_{kj}^i} e_k - T^j_{ki} \overline{e}_k ).
\end{equation*}
In particular, for $t=2$, one can check that
$$ \langle T^s(X,Y), Z\rangle = - \langle T^s(X,Z), Y\rangle $$
for any tangent vector $X$, $Y$, $Z$. So $\nabla^s = \nabla^{(2)}$ is indeed the Hermitian connection with totally skew-symmetric torsion, namely, the Strominger connection. Also, we see that the Chern torsion components $T_{ij}^k$ contain all the torsion information for any $\nabla^{(t)}$. As a consequence, we have
$$ \nabla^s T^s =0 \ \ \Longleftrightarrow \ \ \nabla^sT^c = 0 \ \ \Longleftrightarrow \ \ \nabla^s T^{(t)}=0 $$
for any $t\in {\mathbb R}$. On the other hand, given any $t\neq t'$, one has that $\nabla^{(t)}T^c = 0 $ is not equivalent to   $ \nabla^{(t')}T^c=0$  in general. So for these Gauduchon connections, when we say {\em parallel torsion} it is important to specify which connection makes the torsion parallel.


Next, let us recall Gauduchon's {\em torsion $1$-form} $\eta$ which is
defined to be the trace of $\gamma'$ (\cite{Gauduchon}). Under any
frame $e$, it has the expression:
\begin{equation*}
\eta = \mbox{tr}(\gamma') = \sum_{i,j=1}^n T^i_{ij}\varphi_j = \sum_j \eta_j \varphi_j.
\end{equation*}
Denote by $\omega =\sqrt{-1} \sum_{i,j} g_{i\overline{j}} \varphi_i \wedge \overline{\varphi}_j$ the K\"ahler form of $g$, where $g_{i\overline{j}} = \langle e_i, \overline{e}_j \rangle $. By a direct computation, one gets
\begin{equation}
\partial \,\omega^{n-1} = -2 \ \eta \wedge \omega^{n-1}.
\label{eq:domega}
\end{equation}
Recall that the metric $g$ is said to be {\em balanced} if
$\omega^{n-1}$ is closed. The above identity shows that $g$ is balanced
if and only if $\eta =0$. When $n=2$, $\eta =0$ means $\tau =0$, so
balanced complex surfaces are K\"ahler. But in dimension $n\geq 3$, $\eta $
contains less information than $\tau$. By the structure equations and the first Bianchi identity, one gets the following

\begin{lemma}\label{lemma1}
Under any unitary frame $e$, it holds that
\begin{equation*}
\sqrt{-1} \partial \overline{\partial} \,\omega = \, ^t\!\tau  \overline{\tau} + \, ^t\!\varphi \Theta \overline{\varphi}.
\end{equation*}
\end{lemma}


Next let us consider the curvature tensors. Let $D$ be a linear connection on $M^n$. Its curvature $R^D$ is defined by (\ref{eq:def-curvature}). We will also write it as $R^D_{XYZW}$ for convenience. It is always skew-symmetric with respect to the first two positions, and also skew-symmetric with respect to its last two positions if the connection is {\em metric,} namely, $Dg=0$. When the connection $D$ is {\em Hermitian,} namely, satisfies $Dg=0$ and $DJ=0$, where $J$ is the almost complex structure, then $R^D$ satisfies
$$ R^D(X,Y,JZ,JW) = R^D(X,Y,Z,W) $$
for any tangent vectors $X$, $Y$, $Z$, $W$. Under a type $(1,0)$ frame $e$, the components of the Chern, Strominger, and  Riemannian curvature tensors are
given by
\begin{eqnarray*}
R^c_{i\overline{j}k\overline{\ell}} & = & \sum_{p=1}^n \Theta_{kp}(e_i,
\overline{e}_j)g_{p\overline{\ell}}, \\
R^s_{abk\overline{\ell}} & = & \sum_{p=1}^n \Theta^s_{kp}(e_a,
e_b)g_{p\overline{\ell}}, \\
R_{abcd} & = & \sum_{f=1}^{2n}
\hat{\Theta}_{cf}(e_a,e_b)g_{fd},  \label{formula 16}
\end{eqnarray*}
where $a, \ldots , d, f$ are between $1$ and $2n$, with
$e_{n+i}=\overline{e}_i$. Note that for any Hermitian connection $D$ we have $R^D_{abij} = R^D_{ab\bar{i}\bar{j}} =0$ by the discussion above. For the Riemannian connection $\nabla$, which does not make $J$ parallel in general, $R_{abij}$ may not vanish in general. By
$g_{ij}=g_{\bar{i}\bar{j}}=0$, we get
\begin{eqnarray*}
R_{i\bar{j}k\bar{\ell}} & = & \sum_{p=1}^n
(\Theta_1^{1,1})_{kp}(e_i, \overline{e}_j)g_{p\overline{\ell}}, \ \
R_{ij\bar{k}\bar{\ell}} \ = \  \sum_{p=1}^n
(\Theta_2^{2,0})_{\overline{k}p}(e_i, e_j)  g_{p\overline{\ell}},
\label{formula 17}\\
R_{i\bar{j} \bar{k}\bar{\ell}} & = & R_{\bar{k}\bar{\ell}
i \overline{j} } \ = \ \sum_{p=1}^n (\Theta_2^{1,1})_{\overline{k}p}
(e_i, \overline{e}_j) g_{p\overline{\ell}} \ =
\ \sum_{p=1}^n (\Theta_1^{0,2})_{ip}(\overline{e}_k, \overline{e}_{\ell } )g_{p\overline{j}}  \label{formula 18}, \\
R_{\bar{i}\bar{j}\bar{k}\bar{\ell}} & = & R_{ijk\ell}
\ = \ 0.  \label{formula 19}
\end{eqnarray*}
The last line is because $\Theta_2^{0,2}=0$ by \cite[Lemma 1]{YZ}, a property for general Hermitian metric discovered by Gray in
\cite[Theorem 3.1 on page 603]{Gray}. Note that here we adopted the usual notation for curvature tensor, unlike in \cite{YZ} or \cite{WYZ}, where the first two and last two positions were swapped. As in \cite{WYZ}, the starting point of our computation is the following lemma from \cite[Lemma 7]{YZ}, again note that we have swapped the first two and last two positions for the curvature tensors.

\begin{lemma}  \label{lemma2}
Let $(M^n,g)$ be a Hermitian manifold. Let $e$
be a unitary frame in $M$, then
\begin{eqnarray*}
2T^j_{ik,\overline{\ell}} & = &
R^c_{k\overline{\ell}i\bar{j}} - R^c_{i\bar{\ell}k\bar{j}}
\label{formula 21},\\
R_{k\bar{\ell}ij} \ & = & T^{\ell}_{ij,k} + T^{\ell}_{ri} T^r_{jk} -
T^{\ell}_{rj} T^r_{ik}  \label{formula 22},\\
R_{\bar{j}\bar{\ell}ik} & = &
T^{\ell}_{ik,\bar{j}} - T^j_{ik,\bar{\ell}} +
2T^r_{ik} \overline{T^r_{j\ell}} + T^j_{ri} \overline{ T^k_{r\ell} } +
T^{\ell}_{rk} \overline{ T^i_{rj} } - T^{\ell}_{ri} \overline{ T^k_{rj} } -
T^j_{rk} \overline{ T^i_{r\ell} }   \label{formula 23},\\
R_{k\bar{\ell}i\bar{j}} & = &
R^c_{k\bar{\ell}i\bar{j}} - T^j_{ik,\bar{\ell}} -
\overline{ T^i_{j\ell,\bar{k}} } + T^r_{ik} \overline{ T^r_{j\ell} }
- T^j_{rk} \overline{ T^i_{r\ell} } - T^{\ell}_{ri} \overline{ T^k_{rj} },
\label{formula 24}
\end{eqnarray*}
where the index $r$ is summed over $1$ through $n$, and the index after the comma stands for covariant derivative with respect to the Chern connection $\nabla^c$.
\end{lemma}

Finally, by the same proof of \cite[Lemma 4]{YZ}, we have the following

\begin{lemma}\label{lemma3}
Let $(M^n,g)$ be a Hermitian manifold. For any $p\in M$, there exists a unitary frame $e$ of type $(1,0)$ tangent vectors in a neighborhood of $p$, such that the connection matrix $\theta^s(p)=0$.
\end{lemma}

In other words, one can always choose a local unitary frame such that the connection matrix vanishes at a given point. Of course the same property holds for any Hermitian connection $D$ on $M$, not just the Chern or Strominger connection.

\vs

\section{Strominger K\"ahler-like metrics}\label{BKmetric}

Now let us recall the notion of {\em K\"ahler-like} in describing a metric connection $D$ on a Hermitian manifold $(M^n,g)$:

\vsv
\vsv

\noindent {\bf Definition \cite{AOUV}:} {\em Let $(M^n,g)$ be a Hermitian manifold and $D$ a metric connection on $M$, that is, $Dg=0$. We say that $D$ is {\bf \em K\"ahler-like,} if its curvature tensor $R^D$ obeys the symmetries:
\begin{eqnarray*}
&& R^D(X,Y,Z,W) + R^D(Y,Z,X, W) + R^D(Z, X,Y,W) = 0, \\
&& R^D(X,Y,JZ,JW)= R^D(X,Y,Z,W) = R^D(JX,JY,Z,W)
\end{eqnarray*}
for any tangent vectors $X$, $Y$, $Z$, $W$ in $M^n$.}

\vsv
\vsv

Note that when $D$ is Hermitian, the first equality in the second line above always holds. Also, in terms of complex components, the K\"ahler-like condition simply means that the only possibly non-zero components of $R^D$ are $R^D(x,\overline{y},z,\overline{w})$ where $x$, $y$, $z$, $w$ are type $(1,0)$ complex tangent vectors, and $x$ and $z$ can be interchanged, i.e., $R^D(z,\overline{y},x,\overline{w})= R^D(x,\overline{y},z,\overline{w})$ always holds.

\vsv

As mentioned in the introduction section, this notion was introduced in \cite{YZ} for the Riemannian and Chern connections, following the pioneer works of Gray \cite{Gray} and others. It was generalized to any metric connection by Angella, Otal, Ugarte, and Villacampa in \cite{AOUV}. To prove the AOUV Conjecture (Conjecture 1 in \cite{AOUV}), namely, to show that if the Strominger connection is K\"ahler-like, then the Hermitian metric must be pluriclosed, let us take a closer look at the Strominger K\"ahler-like condition. We begin with the following

\begin{lemma}\label{lemma4}
Let $(M^n,g)$ be a Hermitian manifold, $n\geq 2$. The Strominger connection $\nabla^s$ is K\"ahler-like if and only if
\begin{equation*}
 \ ^t\!\varphi \wedge \Theta^s =0
 \end{equation*}
under any unitary frame $e$.
\end{lemma}

\begin{proof} Note that the above equation  implies that the $(0,2)$ part of $\Theta^s$ is zero, so the $(2,0)$ part is also zero since $\Theta^s$ is skew-Hermitian. For the $(1,1)$ part, write
$$\Theta^s_{k\ell} = \sum_{i,j=1}^n R^s_{i\overline{j}k\overline{\ell}} \varphi_i \wedge \overline{\varphi}_j, $$
we see that the equation in the lemma means that $R^s$ is symmetric with respect to its first and third position. This means that $\nabla^s$ is K\"ahler-like. The converse is also true since one can walk backwards.
\end{proof}

Modifying the results in \cite{WYZ} for the Strominger flat case, we have the following:

\begin{lemma}\label{lemma5}
If a Hermitian manifold $(M^n,g)$ has K\"ahler-like Strominger connection, then under a local unitary frame $e$, the Chern torsion components satisfy
\begin{eqnarray}
T_{ik,\ell }^j - T_{i\ell ,k}^j & = &  2 \sum_{r} (  \ T_{ik}^r T_{r\ell }^j + T_{\ell i}^r T_{rk}^j   +  T_{k\ell }^r T_{ri}^j ) \ \ = \ \ 0, \label{eq:21} \\
T^j_{ik ,\overline{\ell}} + \overline{ T^i_{j\ell ,\overline{k}} } - \overline{ T^{k}_{j\ell ,\overline{i}} } & = & - 2 \sum_r \big( T^r_{ik} \overline{T^r_{j\ell }} + T^j_{ir} \overline{T^k_{\ell r}} + T^{\ell}_{kr} \overline{T^i_{jr }} - T^{\ell}_{ir} \overline{T^k_{jr }} - T^j_{kr} \overline{T^i_{\ell r}} \big), \label{eq:22}
\end{eqnarray}
where the indices after comma means covariant derivatives with respect to $\nabla^s$.
\end{lemma}

\begin{proof}
Fix any $p\in M$ and we want to verify the above identities at $p$. Since both sides are tensors, we may assume without loss of generality that the unitary frame $e$ has vanishing $\theta^s$ at $p$. Since $\gamma'_{ij} = \sum_k T^j_{ik}\varphi_k$, we have $\ ^t\!\gamma' \varphi = - \tau$. So at $p$ it holds that
\begin{equation}\label{eq:structure}
\partial \varphi = - \tau , \ \ \overline{\partial} \varphi =-2\overline{\gamma'}\varphi, \ \ \Theta^s -\Theta = 2d\gamma + 4 \gamma\wedge \gamma.
\end{equation}
Now $0=(\Theta^s)^{2,0} = 2\partial \gamma' + 4\gamma' \gamma'$ lead to the first equality in \eqref{eq:21}. Next, the first Bianchi identity says that $d\tau = -\, ^t\!\theta \tau + \,^t\!\Theta \varphi$. So at $p$ we have $\theta = - 2\gamma$. By taking the $(3,0)$-part of the Bianchi identity, we get $\partial \tau = 2 \, ^t\!\gamma '\tau $, which leads to the second equality in \eqref{eq:21} if $n\geq 3$. Note that when $n=2$, this equality is automatically true, as $i$, $k$, $\ell$ cannot be all distinct.

To prove \eqref{eq:22}, let us write
$$ \Phi = (d\gamma +2\gamma \gamma )^{1,1} = \overline{\partial }\gamma' - \partial\,\overline{^t\!\gamma'}-2\gamma'\, \overline{^t\!\gamma'} - 2\,\overline{^t\!\gamma'} \,\gamma' .$$
At $p$, the Bianchi identity gives
$$\overline{\partial }\tau + 2\overline{\gamma'}\tau = \,^t\!\Theta \varphi = - 2\,^t\!\Phi \varphi,$$
which leads to the equality \eqref{eq:22}. This completes the proof of  Lemma \ref{lemma5}.
\end{proof}

\begin{lemma}\label{lemma6}
If a Hermitian manifold $(M^n,g)$ has K\"ahler-like Strominger connection, then under a local unitary frame $e$, the Chern torsion components satisfy
\begin{eqnarray*}
T_{ik,\ell }^j  & = & 0 \\
0 \ \ \ & = & \sum_{r} (  \ T_{ik}^r T_{r\ell }^j + T_{\ell i}^r T_{rk}^j   +  T_{k\ell }^r T_{ri}^j )  \label{eq:25} \\
T^j_{ik ,\overline{\ell }} & = & - T^{\ell}_{ik ,\overline{j } } \ = \ \overline{  T^i_{j\ell  ,\overline{k }} } \\ \nonumber
& = & -\frac{2}{3}  \sum_r \big( T^r_{ik} \overline{T^r_{j\ell }} + T^j_{ir} \overline{T^k_{\ell r}} + T^{\ell}_{kr} \overline{T^i_{jr }} - T^{\ell}_{ir} \overline{T^k_{jr }} - T^j_{kr} \overline{T^i_{\ell r}} \big)
\end{eqnarray*}
for any $i$, $j$, $k$, $\ell$, where the indices after comma means covariant derivatives with respect to $\nabla^s$.
\end{lemma}

\begin{proof}
From \eqref{eq:21} in Lemma \ref{lemma5}, we know that $T^j_{ik,\ell}$ satisfies $T^j_{ik,\ell} =T^j_{i\ell ,k}$. On the other hand, $T^j_{ik\ell} = - T^j_{ki\ell}$.  Thus
$$ T^j_{ik,\ell} = - T^j_{ki,\ell } = - T^j_{k\ell ,i} =  T^j_{\ell k,i} = T^j_{\ell i,k} = -T^j_{i\ell ,k} = -T^j_{ik,\ell},$$
so $T^j_{ik,\ell}=0$ for all indices. From \eqref{eq:22} in Lemma \ref{lemma5}, if we denote by $P^{j\ell}_{\,ik}$ the five term sigma on the right hand side, that is,
\begin{equation*}
P^{j\ell}_{\,ik} = \sum_r \big( T^r_{ik} \overline{T^r_{j\ell }} + T^j_{ir} \overline{T^k_{\ell r}} + T^{\ell}_{kr} \overline{T^i_{jr }} - T^{\ell}_{ir} \overline{T^k_{jr }} - T^j_{kr} \overline{T^i_{\ell r}} \big),
\end{equation*}
then clearly we have
$$ P^{j\ell}_{\,ik} = - P^{j\ell }_{\,ki} = - P^{\ell j}_{\,ik} = \overline{ P^{ik}_{j\ell } } .$$
By the same proof as in \cite[Lemma 9]{WYZ}, we see that $T^j_{ik ,\overline{\ell }} = - T^{\ell}_{ik ,\overline{j } }  =  \overline{  T^i_{j\ell  ,\overline{k }} }$, so the left hand side of the equality \eqref{eq:22} is $3$ times of $T^j_{ik ,\overline{\ell }}$. This completes the proof of Lemma \ref{lemma6}.
\end{proof}

Lemma \ref{lemma6} gives us nice properties for the Chern torsion components of Strominger K\"ahler-like manifolds just like in the Strominger flat case. However, unlike in the Strominger flat case, we no longer have $\nabla^s$-parallel frames any more, so the Bochner identity argument used in \cite{WYZ} breaks down here, and we have to dig in deeper into the algebraic tanglement of these torsion components and their $\nabla^s$ covariant derivatives. We will first prove the following

\begin{proposition}\label{prop1}
Given a Hermitian manifold $(M^n,g)$ whose $\nabla^s$ is K\"ahler-like,  then the metric $g$ is pluriclosed if and only if $\nabla^s T^c =0$.
\end{proposition}

\begin{proof}  Note that for any $t\in {\mathbb R}$, by definition, $\nabla^{(t)}=\nabla^c + t\gamma$, so the torsion tensor $T^{(t)}=T^c+ t\Gamma$, where $\Gamma (X,Y) = \gamma_XY - \gamma_YX$. Under any unitary frame $e$, the entries of the matrix for $\gamma$ are given by the components $T^j_{ik}$ of $T^c$, so if $\nabla^sT^c=0$, then $\nabla^sT^{(t)}=0$, and vice versa.


Now let us assume that $\nabla^s$ is K\"ahler-like. Then we have $T^j_{ik,\ell }=0$ and $T^j_{ik, \overline{\ell }} = -\frac{2}{3}P_{\,ik}^{j\ell }$ by  Lemma \ref{lemma6}. So $\nabla^sT^c=0$ means $T^j_{ik, \overline{\ell }} =0$, or equivalently, $P_{\,ik}^{j\ell }=0$ for any indices.


Again let us assume that the unitary frame $e$ has vanishing $\theta^s$ at the fixed point $p$. So at the point $p$ we have  $\partial \varphi = -\tau$ and $\overline{\partial} \varphi = -2 \overline{\gamma'}\varphi$. We compute
\begin{eqnarray*}
 \,^t\!\varphi \,\Phi \overline{\varphi}  & = &  \,^t\!\varphi \, ( \overline{\partial} \,\gamma' -  \partial \,\overline{^t\!\gamma'}  - 2\gamma' \,\overline{^t\!\gamma'} - 2 \,\overline{^t\!\gamma'} \gamma'  )  \overline{\varphi}  \\
& = & \sum_{i,j,k,\ell } \varphi_i \left\{ \overline{\partial} \gamma'_{ij} - \partial \overline{\gamma'_{ji}} - 2\sum_r \gamma'_{ir} \overline{\gamma'_{jr}} - 2\sum_r \overline{\gamma'_{ri}} \gamma'_{rj} \right\} \overline{\varphi}_j \\
& = & \sum _{i,j,k,\ell} \frac{1}{4} (Q_{ik}^{j\ell}-Q_{ki}^{j\ell} - Q_{ik}^{\ell j}  +Q_{ki}^{\ell j})\,\varphi_i \varphi_k \overline{\varphi}_j \overline{\varphi}_{\ell},
\end{eqnarray*}
where
\begin{eqnarray*}
Q_{ik}^{j \ell } & = & T_{ik,\overline{\ell}}^j -2 T^{j}_{ir} \overline{T^k_{r\ell }} + \overline{T^i_{j\ell,\overline{k}} } - 2 \overline{T^i_{jr} } T^{\ell }_{rk} + 2 T^r_{ik} \overline{ T^r_{j\ell } } -  2 \overline{T^i_{r\ell} } T^j_{rk}\\
& = & -\frac{4}{3} P_{\,ik}^{j\ell} + 2 T^r_{ik} \overline{ T^r_{j\ell } } + 2 T^{j}_{ir} \overline{T^k_{\ell r}}+ 2 T^{\ell}_{kr} \overline{T^i_{j r}} - 2 T^{j}_{kr} \overline{T^i_{\ell r}}.
\end{eqnarray*}
Here we used the fact that $\overline{ T^i_{j\ell ,\overline{k }} } = T^j_{ik,\overline{\ell}} = -\frac{2}{3}P^{j\ell}_{\,ik}$.  We have
\begin{eqnarray*}
\frac{1}{4} (Q_{ik}^{j\ell}-Q_{ki}^{j\ell} - Q_{ik}^{\ell j}  +Q_{ki}^{\ell j})
& = & - \frac{4}{3} P_{\,ik}^{j\ell} + 2 T^r_{ik} \overline{ T^r_{j\ell } } + \frac{3}{2} \left\{  T^{j}_{ir} \overline{T^k_{\ell r}}+  T^{\ell}_{kr} \overline{T^i_{j r}} -  T^{\ell}_{ir} \overline{T^k_{j r}} -  T^{j}_{kr} \overline{T^i_{\ell r}} \right\} \\
& = & - \frac{4}{3} P_{\,ik}^{j\ell} + 2 T^r_{ik} \overline{ T^r_{j\ell } } + \frac{3}{2} \left\{   P_{\,ik}^{j\ell } - T^r_{ik} \overline{ T^r_{j\ell } } \right\} \\
& = & \frac{1}{6} P_{\,ik}^{j\ell} + \frac{1}{2}T^r_{ik} \overline{ T^r_{j\ell } }.
\end{eqnarray*}
Therefore, we get
\begin{eqnarray*}
\sqrt{-1} \partial \overline{\partial }\,\omega & = & \,^t\!\tau \, \overline{\tau} + \,^t\! \varphi \,\Theta \overline{\varphi}  \ = \ \,^t\!\tau \, \overline{\tau}  - 2  \,^t\! \varphi \,\Phi \overline{\varphi}   \\
& = & \sum_{i,k,j,\ell} \left\{ (T^r_{ik}\overline{T^r_{j\ell}} - \frac{1}{3} P^{j\ell }_{\,ik} - T^r_{ik}\overline{T^r_{j\ell}} \right\}  \varphi_i \varphi_k \overline{\varphi}_j \overline{\varphi}_{\ell} \\
& = & -\frac{1}{3} \sum P^{j\ell }_{\,ik} \,\varphi_i \varphi_k \overline{\varphi}_j \overline{\varphi}_{\ell}.
\end{eqnarray*}
So the metric will be pluriclosed if and only if $P^{j\ell}_{\,ik} =0$, or equivalently, $\nabla^sT =0$. This completes the proof of Proposition \ref{prop1}.
\end{proof}


The main technical part in the proof of Theorem \ref{thm1} is to establish the following:

\begin{proposition}\label{prop2}
Given a Hermitian manifold $(M^n,g)$, if the  Strominger connection $\nabla^s$ is K\"ahler-like, then $\nabla^sT^c=0$, where $T^c$ is the torsion tensor of the Chern connection $\nabla^c$.
\end{proposition}

We will prove this proposition in the next section. Assuming Proposition \ref{prop2}, we are now ready to prove Theorem \ref{thm1}:

\begin{proof}[{\bf Proof of Theorem \ref{thm1} (assuming Proposition \ref{prop2})}] Let $(M^n,g)$  be a  Strominger K\"ahler-like manifold. By Proposition \ref{prop2}, we have $\nabla^sT^c=0$. Then by Proposition \ref{prop1}, we get $\partial \overline{\partial} \omega =0$. Conversely, suppose a Hermitian manifold $(M^n,g)$ is pluriclosed and has $\nabla^sT^c=0$. We want to show that it is Strominger K\"ahler-like, that is, $\,^t\!\varphi \,\Theta^s =0$.


Let us fix a point $p\in M$ and choose a local unitary frame $e$ near $p$ such that $\theta^s$ vanishes at $p$. At the point $p$, we have
$$\theta = -2\gamma , \ \ \partial \varphi = \, ^t\!\gamma' \varphi = -\tau , \ \ \overline{\partial }\varphi = -2\overline{\gamma'} \varphi, \ \   \partial{\tau}=2 ^t\!\gamma' \tau, \ \    \overline{\partial} \,^t\!\tau + 2 \,^t\!\tau \, \overline{ ^t\!\gamma'} = \,^t\!\varphi \, \Theta . $$
As in the proof of Lemma \ref{lemma5}, we have
$$ (\Theta^s)^{2,0} = 2\partial \gamma' + 4\gamma' \gamma' , \ \ \ (\Theta^s)^{1,1} = \Theta + 2\Phi . $$
So the Strominger K\"ahler-like condition means
\begin{equation}\label{eq:KL}
 \partial \gamma' + 2\gamma' \gamma'=0 \ \ \ \mbox{and} \ \ \ \overline{\partial} \,^t\!\tau + 2 \,^t\!\tau \, \overline{ ^t\!\gamma'} = -2 \,^t\!\varphi \, \Phi .
\end{equation}
If we write in components and use the parallelness of $T^c$, the first equality in (\ref{eq:KL}) becomes
$$ \sum_r \left\{ T^j_{ir} T^r_{k\ell } + T^j_{\ell r} T^r_{ik} + T^j_{kr} T^r_{\ell i} \right\} = 0,$$
which is true when $n\geq 3$ since $\partial{\tau}=2 ^t\!\gamma' \tau$, and it is automatically true when $n=2$. To see the second equality of (\ref{eq:KL}), use the fact $\,^t\!\tau = \,^t\!\varphi\,\gamma'$ and $\Phi = (d\gamma + 2\gamma \gamma )^{1,1}$,  we get
$$ \overline{\partial} \,^t\!\tau + 2 \,^t\!\tau \, \overline{ ^t\!\gamma'} + 2 \,^t\!\varphi \, \Phi    =     \, ^t\!\varphi \, ( \overline{\partial }\gamma' - 2 \partial \,\overline{^t\!\gamma'} -2 \gamma' \,\overline{^t\!\gamma'} - 2 \,\overline{^t\!\gamma'} \gamma' ) = \sum S_{ik\overline{\ell}} \,\varphi_i \varphi_k \overline{\varphi}_{\ell }, $$
where
$$ S_{ik\overline{\ell}} = - 2 T^j_{ir} \overline{T^k_{\ell r}} -4 T^{\ell }_{kr} \overline{T^i_{jr}} -2 T^r_{ik} \overline{T^r_{j\ell }} + 2 T^j_{kr} \overline{T^i_{\ell r}}$$
is a column vector whose $j$-th component is given by the right hand side of the above equation. It follows that
$$ \frac{1}{4}(S_{ki\overline{\ell}} - S_{ik\overline{\ell}}) = P_{\,ik}^{j\ell }, $$
thus the second equality of (\ref{eq:KL}) will hold when and only when $P_{\,ik}^{j\ell }=0$, so to prove Theorem \ref{thm1} it suffices to show $P=0$.


By assumption, we have $\nabla^sT^c=0$, and $\overline{\partial}\varphi = -2\overline{\gamma'} \varphi$, therefore
\begin{eqnarray*}
\overline{\partial} \,^t\!\tau \,\overline{\varphi} & = & \overline{\partial} \,(T^{\ell}_{ik} \varphi_i \varphi_k) \, \overline{\varphi}_{\ell}  \ = \ - 2T^{\ell }_{ir} \varphi_i \,\overline{\partial} \varphi_r  \,\overline{\varphi}_{\ell}  \ = \  4 T^{\ell}_{ir} \overline{ T^k_{jr}} \varphi_i\varphi_k \overline{\varphi}_{j} \overline{\varphi}_{\ell} \\
& = & \left\{ T^{\ell}_{ir} \overline{ T^k_{jr}} - T^{\ell}_{kr} \overline{ T^i_{jr}} - T^{j}_{ir} \overline{ T^k_{\ell r}} + T^{j}_{kr} \overline{ T^i_{\ell r}} \right\} \varphi_i\varphi_k \overline{\varphi}_{j} \overline{\varphi}_{\ell} \\
& = & ( T^{r}_{ik} \overline{ T^r_{j\ell }} - P_{\,ik}^{j\ell }) \, \varphi_i\varphi_k \overline{\varphi}_{j} \overline{\varphi}_{\ell}.
\end{eqnarray*}
This leads to the following
\begin{eqnarray*}
\sqrt{-1} \partial \overline{\partial }\, \omega & = & \,^t\!\tau \, \overline{\tau} + \,^t\! \varphi \,\Theta \overline{\varphi}  \ = \ \,^t\!\tau \, \overline{\tau}  + (  \overline{\partial} \,^t\!\tau + 2 \,^t\!\tau \,\overline{^t\!\gamma'} ) \overline{\varphi}   \\
& = &  \,^t\!\tau \, \overline{\tau} + \overline{\partial} \,^t\!\tau  \, \overline{\varphi} - 2\,^t\!\tau \, \overline{\tau} \ = \ \overline{\partial} \,^t\!\tau  \, \overline{\varphi} - \,^t\!\tau \, \overline{\tau}  \\
& = & - P_{\,ik}^{j\ell} \, \varphi_i\varphi_k \overline{\varphi}_{j} \overline{\varphi}_{\ell}.
\end{eqnarray*}
Since the metric is assumed to be pluriclosed, we have $P_{\,ik}^{j\ell }=0$. This completes the proof of Theorem \ref{thm1}, under the assumption that Proposition \ref{prop2} is already established.
\end{proof}

%
%

Our next goal is to prove Theorem \ref{thm2}. Under the coframe $\{\varphi , \overline{\varphi}\}$, the Riemannian (Levi-Civita) connection $\nabla$ has
$$ \nabla \left( \begin{array}{cc} \varphi \\ \overline{\varphi} \end{array} \right) = \, -  \left( \begin{array}{cc} \,^t\!\theta_1 &  \,^t\!\theta_2  \\ \overline{^t\!\theta}_2 &  \overline{^t\!\theta}_1   \end{array} \right)    \left( \begin{array}{cc} \varphi \\ \overline{\varphi} \end{array} \right). $$
Fix any $p\in M$, let us choose a local unitary frame $e$ in a neighborhood of $p$ so that $\theta^s=0$ at $p$. Then at the point $p$, we have $\theta_1=-\gamma$. Let us write $\theta_2 =\beta$, we have
$$
\nabla \varphi_i = - (\theta_1)_{ki} \,\varphi_k - (\theta_2)_{ki} \,\overline{\varphi}_k  \ = \ (\gamma_{ki} ) \varphi_k - ( \beta_{ki} ) \overline{\varphi}_k. $$
From this, we get that
\begin{eqnarray*}
\nabla \eta & = & ( \eta_{i,k } \varphi_{k}  + \eta_{ i, \overline{k} }  \overline{\varphi}_{k}  ) \varphi_i + ( \eta_{k} \gamma_{ik } ) \varphi_i  - ( \eta_{k} \beta_{ik } ) \overline{\varphi}_i,  \\
\nabla \overline{\eta } & = & ( \overline{\eta_{i,\overline{k} } }\, \varphi_{k}  + \overline{\eta_{ i, k }}\,  \overline{\varphi}_{k}  ) \overline{\varphi}_i - ( \overline{\eta}_{k} \overline{\beta}_{ik } ) \varphi_i  + ( \overline{\eta}_{k} \overline{\gamma}_{ik } ) \overline{\varphi}_i,
\end{eqnarray*}
where the index after the comma means covariant derivative with respect to $\nabla^s$. From these identities, we obtain the following

\begin{lemma}\label{lemma7}
On a Hermitian manifold $(M^n,g)$, the real $1$-form $\eta + \overline{\eta}$ is parallel under the Riemannian connection if and only if the following holds:
\begin{eqnarray*}
\eta_{i,k}  & = &  -\eta_r T^r_{ik} \\
\eta_{i,\overline{k}}  & = &  - \overline{\eta}_r T^{k}_{ir} - \eta_r \overline{T^i_{k r}} \label{eq:chi}
\end{eqnarray*}
for any $i$, $k$. Here the index after comma means covariant derivative with respect to $\nabla^s$.
\end{lemma}

In particular, when $n=2$, the right hand sides of the above two formula are always zero, so we get the following corollary:

\begin{lemma}\label{lemma8}
On a Hermitian surface $(M^2,g)$, the real $1$-form $\eta + \overline{\eta}$ is parallel under the Riemannian connection if and only if the torsion tensor $T^c$ is parallel under $\nabla^s$.
\end{lemma}

Now we are ready to prove Theorem \ref{thm2} stated in the introduction:

\begin{proof}[{\bf Proof of Theorem \ref{thm2}}] Note that when $n=2$, the only component of $P$ is $ P^{12}_{\,12}$, which equals to $|T|^2-2|\eta|^2 $ and is always zero. So  for Hermitian surfaces, Strominger K\"ahler-like is equivalent to $\nabla^sT^c=0$, which is equivalent to Vaisman by the above Lemma. Such a surface is always locally conformally K\"ahler and pluriclosed. This completes the proof of Theorem \ref{thm2}.
\end{proof}

\vs

\section{The parallelness of the torsion}\label{pll_tor}

In this section, we will prove Proposition \ref{prop2}, the main technical result of this article. First let us focus on Gauduchon's torsion $1$-form $\eta$, which is defined by $\eta = \sum_i \eta_i \varphi_i$ where $\eta_i=\sum_{k} T^k_{ki}$. In the third equation of Lemma \ref{lemma6}, if we let $i=j$ and $k=\ell$ and sum them up from $1$ to $n$, we get
$$ \sum_i \eta_{i,\overline{i}}  = \frac{2}{3} \big( \sum_{i,j,k}|T^j_{ik}|^2 - 2 \sum_{i}|\eta_i|^2 ) \big)  = \frac{2}{3} \big( |T|^2 - 2|\eta|^2  \big) .$$
By (\ref{eq:structure}), we have
$ \overline{\partial } \eta = - \sum_{i,j=1}^n (\eta_{i,\overline{j}} + 2 \sum_p \eta_p \overline{T^i_{jp}}) \varphi_i\wedge \overline{\varphi_j}$, hence
$$  \sqrt{-1} \ \overline{\partial } \eta \wedge \omega^{n-1} = - \sum_i (\eta_{i,\overline{i}} + 2|\eta_i|^2) \frac{\omega^n}{n},$$
where $\omega$ is the K\"ahler form of the metric of $M^n$. On the other hand, by $(\ref{eq:domega})$, we have
$$ \partial \overline{\partial }\omega^{n-1} = 2 (\overline{\partial } \eta + 2 \eta \wedge \overline{\eta })\wedge \omega^{n-1}.$$
Combining the above  identities, we get the following:

\begin{lemma} \label{lemma9}
On a Strominger K\"ahler-like manifold  $(M^n,g)$, it holds that
\begin{equation*}
- \sqrt{-1} \  \partial \overline{\partial }\omega^{n-1} = \frac{2}{n} (\sum_i \eta_{i,\overline{i}}) \ \omega^n = \frac{4}{3n} ( |T|^2 - 2|\eta |^2 ) \ \omega^n.
\end{equation*}
In particular, when $M$ is compact, one has $ \int_M (|T|^2 -2|\eta|^2)\omega^n = 0$, so the metric $g$ cannot be balanced unless it is K\"ahler.
\end{lemma}

From now on we will denote by $|\eta|^2=\sum_i |\eta_i|^2$ and $|T|^2=\sum_{i,j,k} |T^j_{ik}|^2$ under any unitary frame. Note that under the frame $\{ e, \overline{e}\}$, the torsion tensor $T^c$ of the Chern connection takes the form
$$ T^c(e_i, e_j) = 2\sum_k T^k_{ij} e_k, \ \ \ \ T^c(e_i, \overline{e_j})= 0, \ \ \ \ T^c(\overline{e_i}, \overline{e_j})=2\sum_k \overline{T^k_{ij}} \overline{e_k}, $$
so $|\!|T^c|\!|^2= 8\sum_{i,j,k} |T^k_{ij}|^2 = 8|T|^2$. When $n=2$, the torsion tensor has only two components:
$T^1_{12}$ and $T^2_{12}$, while the Gauduchon $1$-form has coefficients
$\eta_1 = -T^2_{12}$ and $\eta_2 = T^1_{12}$, so we always have
$|T|^2=2|\eta|^2$ when $n=2$. That is, a Strominger K\"ahler-like surface is always pluriclosed.

\vsv

Next, let us introduce the following notations:
\begin{equation*}
 A_{k\overline{\ell }} = \sum_{r,s} T^r_{sk} \overline{ T^r_{s\ell } } , \ \ \ \     B_{k\overline{\ell }} = \sum_{r,s} T^{\ell }_{rs} \overline{ T^k_{rs } } , \ \ \ \   C_{ik} = \sum_{r,s} T^r_{si} T^s_{rk},  \ \ \ \   \phi^{\ell }_k = \sum_r \overline{\eta}_r T^{\ell }_{kr}.
 \end{equation*}
Clearly, $C$ is symmetric, while $A$, $B$ are Hermitian symmetric. By taking trace of the identities in Lemma \ref{lemma6}, we get the following:

 \begin{lemma}\label{lemma10}
 Let $(M^n,g)$ be a Hermitian manifold that is Strominger K\"ahler-like. Then
 \begin{eqnarray}
 && \eta_{i,k}=0, \ \ \ \ \ \ \ \ \sum_r \eta_r T^r_{ik} = 0, \ \ \ \ \  \sum_r \eta_{r, \overline{r}} = \frac{2}{3} ( |T|^2 - 2 |\eta |^2), \label{eq:eta_1}\\
&& \eta_{k, \overline{\ell }} \ = \  \overline{\eta_{\ell , \overline{k }} }\ = \ - \frac{2}{3} S_{k\overline{\ell }} \ :=  \ - \frac{2}{3} ( \phi_k^{\ell} + \overline{   \phi^k_{\ell }  }   -  B_{k\overline{\ell }} ) . \nonumber
 \end{eqnarray}
 for any $i$, $k$, $\ell$, where the index after comma means covariant derivative in $\nabla^s$.
 \end{lemma}
Note that the quantity $S_{k\overline{\ell}}$ above is simply $\sum_i P^{i\ell }_{\,ik}$. Next let us derive some commutativity formula. Fix a point $p\in M$, and let $e$ be a local unitary frame such that $\theta^s$ vanishes at $p$. Since $\nabla e_i = \nabla^se_i - \gamma e_i + (\overline{\theta_2})_{ij}e_j$, at $p$ we have
\begin{equation*}\label{eq:liebracket}
[e_k, e_j]\ = \  \nabla_{e_k} e_j - \nabla_{e_j}e_k \ = \  - \gamma_{e_k}e_j + \gamma_{e_j} e_k \ = \ 2 \sum_r T^r_{kj} e_r.
\end{equation*}
Again at the point $p$, we compute
\begin{eqnarray*}
&& \eta_{i, \overline{j}}  \ = \ \overline{e}_j (\eta_i) - \sum_r \eta_r \langle \nabla^s_{\overline{e}_j} e_i , \overline{e}_r \rangle, \\
&& \eta_{i, \overline{j}\,\overline{k} }  \ = \ \overline{e}_k (\eta_{i, \overline{j}} ) \ = \ \overline{e}_k  (\overline{e}_j (\eta_i)) -  \sum_r \eta_r \langle \nabla^s_{ \overline{e}_k} \nabla^s_{\overline{e}_j} e_i , \overline{e}_r \rangle, \\
&& \eta_{i, \overline{j}\,\overline{k} } - \eta_{i, \,\overline{k}\overline{j} } \ = \ [ \overline{e}_k, \overline{e}_j ] \,\eta_i - \sum_r \eta_r \langle R^s_{\overline{e}_k \overline{e}_j} e_i, \overline{e}_r \rangle.
\end{eqnarray*}
The curvature term is $\Theta^s_{ir}( \overline{e}_k, \overline{e}_j)$, which equals to $0$ since $(\Theta^s)^{0,2}=0$, so we get the following
\begin{equation*}\label{eq:commute}
\eta_{i, \overline{j}\,\overline{k} } - \eta_{i, \,\overline{k}\overline{j} } \  =  \    2 \sum_r \overline{T^{r}_{kj}} \, \eta_{i, \overline{r}}.
\end{equation*}

\begin{lemma} \label{lemma11}
On a Strominger K\"ahler-like manifold $(M^n,g)$, the equality  $\sum_k \eta_{k, \overline{\ell}} \, \overline{\eta}_k =0$ holds for any index $\ell$. In particular, $|\eta |^2$ is a constant.
\end{lemma}

\begin{proof}
By Lemma \ref{lemma10}, we have
\begin{equation*}
-\frac{3}{2} \sum_k \eta_{k, \overline{\ell}} \, \overline{\eta}_k  = \sum_{r,k} \overline{\eta}_r \overline{\eta}_k T^{\ell }_{kr} + \sum_{r,k} \eta_r \overline{ T^k_{\ell r} \eta_k } - \sum_{r,s,k} T^{\ell }_{rs} \overline{ T^k_{rs} \eta_k } = 0,
\end{equation*}
since $T^{\ell }_{kr} = - T^{\ell}_{rk}$ and the second equality of \eqref{eq:eta_1} hold. This together with $\eta_{k,\ell }=0$ implies that $|\eta |^2_{,\overline{\ell} }=0$ for any $\ell$, hence $|\eta |^2$ is a constant.
\end{proof}

From now on, we will use the Einstein convention on indices, namely, any index appearing twice is summed up from $1$ to $n$. By taking the covariant derivative in $\ell$ for the identity in Lemma \ref{lemma11} and summing it up, we get
\begin{equation}\label{sum}
|\eta_{k, \overline{\ell}} |^2 + \eta_{k, \overline{\ell} \ell } \, \overline{\eta}_k = 0.
\end{equation}
For the first term, we have
\begin{equation}\label{square}
\frac{9}{4} |\eta_{k, \overline{\ell}} |^2  = | \phi + \phi^{\ast} - B|^2 = |B|^2 + |\phi +\phi^{\ast }|^2 - 2 (\phi B + \overline{\phi B}),
\end{equation}
where $\phi^{\ast}= \,^t\!\overline{\phi}$ and $\phi B = \sum_{k,\ell } \phi^{\ell }_k B_{\ell \overline{k}}$. By the commutativity formula (\ref{eq:commute}), it follows that
\begin{eqnarray*}
\eta_{k, \overline{\ell} \ell } & = & ( \overline{   \eta_{\ell , \overline{k} } } )_{\!,\ell }  \ = \ \overline{   \eta_{\ell , \overline{k} \,\overline{\ell}  }       }     \ = \ \overline{   \eta_{\ell , \overline{\ell} \,\overline{k}  }  - 2 \overline{T^r_{k\ell }}\, \eta_{\ell, \overline{r}}       } \\
& = & \frac{2}{3}( |T|^2 - 2|\eta |^2)_{\!,k} - 2 T^r_{k\ell } \, \eta_{r, \overline{\ell }} \ = \ \frac{2}{3} |T|^2_{,k} - 2 T^r_{k\ell } \, \eta_{r, \overline{\ell }}   \ \ \ \ \ \\
& = & \frac{2}{3} T^i_{jr} \overline{ T^i_{jr, \overline{k}} } - 2 \,T^r_{k\ell } \,\eta_{r, \overline{\ell} } \ = \ \frac{2}{3} T^i_{jr} \, T^j_{ik,\overline{r}} - 2 \,T^r_{k\ell } \,\eta_{r, \overline{\ell} }.
\end{eqnarray*}
Now use the formula for $\nabla^s$-covariant derivatives of the torsion, the second term in (\ref{sum}) becomes
\begin{eqnarray*}
\frac{9}{4} \eta_{k, \overline{\ell} \ell }\overline{\eta}_k  & = &  -  T^i_{jr} \overline{\eta}_k ( T^s_{ik} \overline{T^s_{jr} } +  T^r_{ks} \overline{T^i_{js} } - T^j_{ks} \overline{T^i_{rs} }  )   +3  T^r_{k\ell } \overline{\eta}_k ( \phi_r^{\ell} + \overline{ \phi^r_{\ell } } - B_{r\overline{\ell}}   )  \nonumber \\
& = & - \phi^s_i B_{s\overline{i}} + \phi^r_sA_{r\overline{s}} + \phi^j_sA_{j\overline{s}} - 3 \phi^r_{\ell} ( \phi_r^{\ell} + \overline{ \phi^r_{\ell } } - B_{r\overline{\ell}}   ) \nonumber \\
& = & 2 \phi B + 2\phi A - 3 |\phi |^2 - 3 \phi^i_k \phi^k_i.
\end{eqnarray*}
For simplicity, let us denote the last term by $\phi \cdot \phi$. We have
\begin{equation}\label{second}
\frac{9}{4} Re\, (\eta_{k, \overline{\ell} \ell }\overline{\eta}_k ) = 2 Re\, (\phi B) + 2 Re\, (\phi A) - \frac{3}{2} |\phi +\phi^{\ast }|^2,
\end{equation}
where we have used the fact that
$$ |\phi + \phi^{\ast} |^2 = (\phi^i_k + \overline{ \phi^k_i} ) ( \overline{ \phi^i_k  } +\phi^k_i )  = 2|\phi|^2 + 2Re\, ( \phi \cdot \phi ).$$
Plugging (\ref{square}) and (\ref{second}) into (\ref{sum}), we obtain
\begin{equation}\label{sumzero}
|B|^2 + 2 Re\, (\phi A) - 2Re\, (\phi B) - \frac{1}{2} |\phi +\phi^{\ast }|^2 =0.
\end{equation}

Now let us focus  on the terms $\phi A$ and $\phi B$. We start from the identity $ \eta_r T^r_{ik}=0$ in Lemma \ref{lemma10}. Taking covariant derivative in $\overline{\ell}$, we get
$$ (\phi^{\ell}_r + \overline{\phi^r_{\ell} } - B_{r \overline{\ell}})T^r_{ik} + \eta_r (   T^s_{ik}\overline{T^s_{r\ell } } - T^{\ell}_{is}\overline{T^k_{rs } } + T^{\ell}_{ks}\overline{T^i_{rs } }  )=0,$$
or equivalently,
\begin{equation}\label{central}
 (\phi^{\ell}_r  - B_{r \overline{\ell}} )T^r_{ik} +  T^{\ell }_{is} \overline{ \phi^k_s} - T^{\ell }_{ks} \overline{ \phi^i_s} =0.
 \end{equation}
Multiplying the above by $\overline{\eta}_k$ and summing up $k$, we get
$$ \phi_i^r B_{r\overline{\ell}} = \phi_i^r\phi_r^{\ell} + \phi_s^{\ell } \overline{\phi_s^i}.$$
Let $\ell =i$ and sum up, we get
\begin{equation*}\label{B}
\phi B = \phi \cdot \phi + |\phi |^2.
\end{equation*}
Taking the real parts, we get
\begin{equation}\label{Breal}
Re\, (\phi B )= Re\, (\phi \cdot \phi )+ |\phi |^2 = \frac{1}{2} |\phi + \phi^{\ast }|^2.
\end{equation}
Now if we multiply on (\ref{central}) by $\overline{T^{\ell }_{ik}}$ and sum up all indices, it yields that
\begin{equation*}\label{BA}
\phi B -|B|^2 + 2 \overline{\phi A} = 0,
\end{equation*}
and by taking the real part, we obtain
\begin{equation*}\label{BAreal}
Re\,(\phi B) -|B|^2 + 2 Re\, (\phi A) = 0.
\end{equation*}
Subtracting that from (\ref{sumzero}), we get
$$ 2|B|^2 - 3 Re\, (\phi B) - \frac{1}{2}|\phi + \phi^{\ast }|^2 = 0,$$
and compare this last equality with (\ref{Breal}), we get
\begin{equation*}\label{etafinal}
|B|^2 = 2 Re\,(\phi B)=  |\phi + \phi^{\ast }|^2 \ \ \ \mbox{and} \ \ \ 2 Re\, (\phi A) = Re\, (\phi B).
\end{equation*}
Plug them into (\ref{square}), we see that $\eta_{k,\overline{\ell}}=0$ for any $k$, $\ell$. Thus we have proved the following:

\begin{lemma} \label{lemma12}
If the Hermitian manifold $(M^n,g)$ is Strominger K\"ahler-like, then its torsion $1$-form $\eta$ is parallel in $\nabla^s$.
\end{lemma}


Our next goal is to show that, under the Strominger K\"ahler-like assumption, the torsion tensor $T^c$ will also be parallel with respect to the Strominger connection $\nabla^s$, which will complete the proof of Proposition \ref{prop2}.


Let $(M^n,g)$ be a Hermitian manifold which is Strominger K\"ahler-like. From our earlier discussion, we already established that $\eta_{k, \overline{\ell }}=0$ holds for all indices. In particular, $|T|^2=2|\eta|^2$ is a constant. It is also known that $B=\phi + \phi ^{\ast}$.

\begin{lemma} \label{lemma13}
On any Strominger K\"ahler-like manifold $(M^n,g)$, the tensors $\phi$ and $B$ are parallel with respect to $\nabla^s$.
\end{lemma}

\begin{proof}
By Lemma \ref{lemma6}, $T^j_{ik,l}=0$, so $\phi^j_{i,k}=0$, and it suffices to show that $\phi^j_{i, \overline{\ell }} =0$ for all indices, or equivalently, $\sum_k \overline{\eta}_k T^j_{ik, \overline{\ell }}=0$. Since $T^j_{ ik, \overline{\ell } } = - \overline{ T^k_{j\ell , \overline{i}}  } $, and $\sum_k \eta_k T^k_{j\ell} =0$, we see that $\phi$, hence $B$, is parallel in $\nabla^s$.
\end{proof}

We will also need the the following commutativity formula for $T$. Since $R^s_{\bar{k}\bar{\ell} \ast \ast }=0$, we have
$$  T^i_{j\ell , \overline{k} \,\overline{\ell }} \ = \ T^i_{j\ell , \overline{k} \,\overline{\ell }} - T^i_{j\ell , \overline{\ell } \,\overline{k} } \ = \ [\overline{e}_{\ell} , \overline{e}_k] \,T^i_{j\ell } \ = \ -2 \,\overline{T^p_{k\ell } } \, T^i_{j\ell , \overline{p}}\,.$$
The first equality holds because $T^i_{j\ell, \overline{\ell}} = - T^{\ell }_{j\ell , \overline{i}} =  \eta_{j, \overline{i}} =0$. Here and below we always use the Einstein convection for indices, namely, any repeated index is summed up. From the above, we get
\begin{equation}\label{Tdoubleprime}
\overline{T^j_{ik}} T^j_{ik, \overline{\ell} \ell } \ = \  \overline{T^j_{ik}} \overline{T^i_{j\ell , \overline{k} \, \overline {\ell }}}    \ = \   -2 \overline{T^j_{ik}} T^p_{k\ell } \overline{ T^i_{j\ell , \overline{p}}  } \ = \   -2 \overline{T^j_{ik}} T^p_{k\ell } T^j_{ip, \overline{\ell} } \ = \ -2 T^p_{k\ell} A_{p\overline{k}, \overline{\ell}}.
\end{equation}
We will use this to deduce the following

\begin{lemma}\label{lemma14}
For a Strominger K\"ahler-like manifold $(M^n,g)$, the torsion tensor $T$ will be parallel in $\nabla^s$ if the following holds:
\begin{equation*}\label{Aparallel}
\sum_{p,k,\ell} T^p_{k\ell} A_{p\overline{k}, \overline{\ell}} =0.
\end{equation*}
\end{lemma}

\begin{proof}
Since $|T|^2$ is a constant, by taking derivative in $\overline{\ell}$, we get
$$ \overline{T^j_{ik} } \, T^j_{ik, \overline{\ell }} =0.$$
Taking covariant derivative in $\ell $ again and summing up $\ell$, we have
\begin{equation*}\label{sum2}
|T^j_{ik, \overline{\ell}} |^2 + \overline{T^j_{ik} } \,T^j_{ik, \overline{\ell } \ell } = 0.
\end{equation*}
By (\ref{Tdoubleprime}), the second term on the left is equal to $-2 \sum_{p,k,\ell} T^p_{k\ell} A_{p\overline{k}, \overline{\ell}}$, whose vanishing would imply  the vanishing of the square term, which means $\nabla^sT^c=0$, so the lemma is proved.
\end{proof}

Now we are finally ready to prove Proposition  \ref{prop2}.

\begin{proof}[{\bf Proof of Proposition \ref{prop2}}]
For any index $k$ we have
$$ 0 = |T|^2_{,\overline{k}} = \overline{T^i_{j\ell }} \, T^i_{j\ell ,\overline{k}} = \overline{T^i_{j\ell }} \, \overline{T^j_{ik, \overline{\ell}} } \ . $$
Taking conjugate and using the derivative formula in Lemma \ref{lemma6}, we get
\begin{eqnarray*}
0 & = & -\frac{3}{2} T^i_{j\ell } \, T^j_{ik, \overline{\ell}}  \ = \ T^i_{j\ell }  P_{\,ik}^{j\ell } \\
& = & T^i_{j\ell } \left\{  T^r_{ik}\, \overline{T^r_{j\ell }}  + T^j_{ir}\, \overline{T^k_{\ell r}} +  T^{\ell}_{kr}\, \overline{T^i_{j r}}  -T^{\ell}_{ir}\, \overline{T^k_{j r}}  -  T^j_{kr}\, \overline{T^i_{\ell r}}  \right\} \\
& = & B_{r\overline{i}} T^r_{ik} + C_{\ell r} \overline{ T^k_{\ell r}} +A_{\ell \overline{r}} T^{\ell}_{kr} + C_{jr}\overline{T^k_{jr} } + A_{j\overline{r}} T^j_{kr} \\
& = & B_{r\overline{i}} T^r_{ik} + 2 A_{j\overline{r}} T^j_{kr},
\end{eqnarray*}
where the last equality is due to the fact that $C_{jr} = C_{rj}$ while $T^k_{jr} = -T^k_{rj}$. Taking the derivative in $\overline{k}$ and summing up $k$, since $B$ is parallel, and $T^a_{bk, \overline{k}} = - T^k_{bk, \overline{a}} = \eta_{b, \overline{a}} =0$, we get
$$ A_{j\overline{r}, \overline{k}} T^j_{kr} = 0.$$
By Lemma \ref{lemma14}, this implies  that $\nabla^sT=0$, so we have completed the proof of Proposition \ref{prop2}.
\end{proof}


\vs

\section{The proof of Theorem \ref{thm3}}

In this section, we will discuss some general properties for compact Strominger K\"ahler-like manifolds, and prove Theorem \ref{thm3} stated in the introduction. We begin with the following
\begin{lemma}\label{lemma15}
On a Strominger K\"ahler-like manifold $(M^n,g)$, it holds that
\begin{equation}
\partial \eta =0 , \ \ \ \ \ \overline{\partial} \eta = -2 \sum_{i,j} \overline{\phi^i_j} \, \varphi_i \,\overline{\varphi}_j
\end{equation}
where $\varphi$ is any unitary coframe.
\end{lemma}

\begin{proof}
 When the metric $g$ is K\"ahler, we have $T=0$, $\eta=0$, and $\phi=0$, so the above identities hold. Let us now  assume that $g$ is not K\"ahler. Under any local unitary frame $e$ and dual coframe $\varphi$, let us write $\eta = \sum_i \eta_i \varphi_i$. We observe  that $X = \sum_i \overline{\eta}_i e_i$ is a globally defined vector field, with $|X|^2=\sum_i |\eta_i|^2 = |\eta |^2$ being a positive constant since $g$ is not K\"ahler. Clearly, $\nabla^s X=0$.

Let us choose our local unitary frame $e$ so that $e_n=\frac{X}{|X|}$.  Since $\nabla^s e_n=0$, the matrix of connection $\nabla^s$ under $e$ takes the form
$$ \theta^s = \left[ \begin{array}{cc} \ast  & 0 \\ 0 & 0 \end{array} \right], $$
where $\ast $ is the $(n-1)\times (n-1)$ block. Also, under this frame $e$, we have $\eta_1=\cdots = \eta_{n-1}=0$ and $\eta_n=\lambda $, where $\lambda =|\eta|>0$  is a constant. By the structure equation, $\partial \varphi = 2 \,^t\!\gamma' \varphi - \,^t\!(\theta^s)' \varphi  + \tau$. Since $\theta^s_{jn}=0$ and $T^n_{\ast \ast }=0$ by the second equality of \eqref{eq:eta_1}, we have
$$ \partial \eta = \lambda \partial \varphi_n = \lambda  (2 \gamma'_{jn} \varphi_j + \tau_n)=0.$$
Similarly, by the structure equation $\overline{\partial }\varphi = \overline{\theta'}\varphi = \overline{( (\theta^s)' \!- \!2\gamma' )} \,\varphi $, we get
$$ \overline{\partial } \eta = \lambda \overline{\partial }\varphi_n = -2\lambda\, \overline{T^i_{nj}} \,\overline{\varphi}_{j} \varphi_i  = -2 \overline{ \phi^i_j} \,\varphi_i \overline{\varphi}_j .$$
This completes the proof of the lemma.
\end{proof}


\begin{proof}[{\bf Proof of Theorem \ref{thm3}}] Let $(M^n,g)$ be a compact Hermitian manifold which is Strominger K\"ahler-like. Assume that $g$ is not K\"ahler. We want to show that $M$ does not admit any balanced metric.

By the lemma above, we have an expression of $\overline{\partial} \eta$ in terms of the tensor $\phi$. Since $B=\phi +\phi^{\ast }$ is Hermitian, we may rotate our unitary frame $e$ to assume that $B$ is diagonal: $B_{i\overline{j}} = b_i\delta_{ij}$, where $b_i = \sum_{j,k} |T^i_{jk}|^2 \geq 0$. Let $\psi = \phi - \frac{1}{2}B$, then $\psi + \psi^{\ast }=0$ and we have
$$ \overline{\partial} \eta = -\sum_i b_i \varphi_i \overline{\varphi}_i - 2 \sum_{i,j} \overline{\psi^i_j} \varphi_i \overline{\varphi}_j. $$
Now suppose that $g_0$ is a balanced metric on $M^n$. Locally under the $g$-unitary frame $e$, we may write
$$ \omega_0^{n-1} = \sum_{i,j} H_{i\overline{j}} \, \widehat{ \varphi_i \overline{\varphi}_j},  $$
where $H$ is a positive definite Hermitian matrix and
$$ \sqrt{-1} \varphi_i \overline{\varphi}_j \wedge \widehat{ \varphi_k \overline{\varphi}_{\ell} } = \delta_{ik} \delta_{j\ell }\,\omega^n. $$
This gives us
$$ \sqrt{-1} \, \overline{\partial} \eta \, \omega_0^{n-1} = \left( -\sum_i b_i H_{i\overline{i}} - 2 \sum_{i,j} \overline{\psi^i_j}  H_{i\overline{j}} \right) \omega^n = (-x - 2 y)\,\omega^n. $$
Note that $\sigma := \sqrt{-1}\sum_{i,j} B_{i\bar{j}} \varphi_i \wedge \overline{\varphi}_j$ is a globally defined non-negative $(1,1)$-form on $M$, and $\sigma \wedge \omega_0^{n-1}= x \omega^n$, so $x$ hence $y$ is globally defined.
Since $\psi$ is skew-Hermitian, we have
$$ y = tr(H \overline{\psi}) = tr (\,^t\! (H\overline{\psi} )) = tr (\psi^{\ast }\,^t\! H) = - tr(\psi \overline{H}) = - tr (\overline{H} \psi ) = -\overline{y},$$
which implies that $y$ is pure imaginary, while $x$ is clearly real and nonnegative. Since $g_0$ is balanced, $d(\omega_0^{n-1})=0$, so $\, \overline{\partial}\eta \,\omega_0^{n-1}$ is exact, whose integral over $M^n$ is zero. By taking its real part, we know that the integral of $x$ over $M$ is zero, which forces $x$ to be identically zero. This leads to $b_i=0$ for each $i$, or equivalently, $T=0$, which contradicts with the assumption that $g$ is not K\"ahler. Thus we have completed the proof of Theorem \ref{thm3}.
\end{proof}

Recall that a Hermitian manifold $(M^n,g_0)$ is called {\em strongly Gauduchon,} if $\partial \omega_0^{n-1}$ is $\overline{\partial}$-exact, where $\omega_0$ is the K\"ahler form of $g_0$. This condition was introduced by Popovici \cite{Popovici}, and has been studied extensively in deformation and modification stability problems.

In Theorem \ref{thm3}, if the metric $g_0$ is only assumed to be strongly Gauduchon, then the same argument works. To be more precise, we have the following

\begin{theorem}\label{thm4}
If a compact Hermitian manifold $(M^n,g)$ is Strominger K\"ahler-like and $g$ is not K\"ahler, then $M^n$ does not admit any strongly Gauduchon metric.
\end{theorem}

\begin{proof}
Assume in the contrary that $M^n$ admits a  strongly Gauduchon metric $g_0$, then by definition there exists a $(2n-2)$-form $\Omega$ such that $\overline{\partial }\omega_0^{n-1}=\partial \Omega$. So we have
$$ \int_M \overline{\partial }\eta \wedge \omega_0^{n-1} = \int_M \eta \wedge \overline{\partial }\omega_0^{n-1} =\int_M \eta \wedge \partial \Omega = \int_M \partial \eta \wedge \Omega =0,$$
since $\partial \eta =0$ by Lemma \ref{lemma15}. Therefore the same proof of Theorem \ref{thm3} will go through in this case.
\end{proof}

\begin{remark}
Note that pluriclosed metric and strongly Gauduchon metric could co-exist on a compact non-K\"ahlerian manifold. In \cite{OUV}, A. Otal, L. Ugarte, R. Villacampa established existence of compact non-K\"ahler manifolds admitting a
Hermitian metric that satisfies both the pluriclosed and the strongly Gauduchon conditions. So in the above theorem, the parallelness of torsion has played a key role.
\end{remark}

Furthermore, Strominger K\"ahler-like metrics are necessarily Gauduchon, namely

\begin{proposition}\label{prop3}
Let $(M^n,g)$ be a Hermitian manifold that is Strominger K\"ahler-like. Then the metric is Gauduchon in the sense that $\partial \overline{\partial} (\omega^{n-1}) =0$.
\end{proposition}

\begin{proof}
Since we already proved that the torsion tensor is parallel under the Strominger connection $\nabla^s$, by Lemma \ref{lemma10}, we see that
$2|\eta |^2 -|T|^2=0$. Hence by Lemma \ref{lemma9}, we get $\partial \overline{\partial} (\omega^{n-1}) =0$. This means that $g$ is always Gauduchon.
\end{proof}

\begin{remark}
Combining this with the pluriclosedness, we know that when $n\geq 3$, any Strominger K\"ahler-like manifold $(M^n,g)$ always satisfies the following pointwise identity:
\begin{equation*}
\partial \omega \wedge \overline{\partial } \omega \wedge  \omega^{n-3} =0.
\end{equation*}
\end{remark}

The above equality along with the pluriclosedness of $g$ immediately give us  the following
\begin{equation*}
\partial  \overline{\partial }( \omega^{k}) \wedge  \omega^{n-k-1} =0 , \ \ \ \ \ \forall \ 1\leq k\leq n-1
\end{equation*}
That is, any Strominger K\"ahler-like metric $g$ is {\em $k$-Gauduchon} for any $k$, a condition introduced by Fu-Wang-Wu in \cite{FuWangWu}. Another distinctive property about Strominger K\"ahler-like manifolds is the following:

\begin{proposition} \label{prop4}
If $(M^n,g)$ is a compact Strominger K\"ahler-like manifold and $g$ is not K\"ahler, then the Dolbeault cohomology group $H^{0,1}_{\overline{\partial }} (M)\neq 0$.
\end{proposition}

\begin{proof}
This is because by Lemma \ref{lemma15}, $\overline{\partial } \overline{\eta } =0$, and if $\overline{\eta} = \overline{\partial }f$ for some smooth function $f$ on $M^n$, then
$$ \int_M \partial \overline{\eta} \wedge \omega^{n-1} = \int_M \partial \overline{\partial }f \wedge \omega^{n-1} =0$$
as $g$ is Gauduchon. This leads to the vanishing of the torsion, hence $g$ must be K\"ahler, a contradiction.
\end{proof}

\vsv
\vsv

\noindent\textbf{Acknowledgments.} The first named author is grateful to the Mathematics Department of Ohio State University for the nice research environment and the warm hospitality during his stay. The second named author would like to thank his collaborators Gabriel Khan, Qingsong Wang and Bo Yang for their previous joint works, which laid the foundation for the computation carried out in the present paper. We are also very grateful to the referee for the exceptionally long list of suggestions for improvement and typo corrections, which enhanced the readability of the paper.

\vs


\begin{thebibliography}{99}
\bibitem {AI} B. Alexandrov and S. Ivanov,  \emph{Vanishing theorems on Hermitian manifolds,}
Diff. Geom. Appl. {\bf 14} (2001),  251-265.

\bibitem  {AOUV}  D. Angella, A. Otal, L. Ugarte, R. Villacampa,  \emph{On Gauduchon connections with K\"ahler-like curvature,} arXiv:1809.02632v2, to appear in Comm. Anal. Geom.

\bibitem  {Belgun} F. Belgun, \emph{On the metric structure of non-K\"ahler complex surfaces,}  Math. Ann. {\bf 317} (2000), 1-40.

\bibitem  {Bismut} J.-M. Bismut, \emph{A local index theorem for non-K\"ahler manifolds,}  Math. Ann. {\bf 284} (1989), no. 4, 681-699.

\bibitem {Boothby} W. Boothby, \emph{Hermitian manifolds with zero curvature,}
Michigan Math. J. {\bf 5} (1958), no. 2, 229-233.

\bibitem {CHSW} P. Candelas, G.T. Horowitz, A. Strominger, E. Witten, \emph{ Vacuum configurations for superstrings,} Nucl. Phys. B, {\bf 258} (1985), 46-74.

\bibitem {EFV} N. Enrietti, A. Fino, and L. Vezzoni, \emph{Tamed symplectic forms and strong K\"ahler with torsion metrics, } J. Symplectic Geom. {\bf 10} (2012), no.2, 203-223.


\bibitem {FY} T. Fei and S.T. Yau, \emph{ Invariant solutions to the Strominger system on complex Lie groups and their quotients,} Comm. Math. Phys. {\bf 338} (2015), no.3, 1-13.





\bibitem {FinoTardini} A. Fino and N. Tardini,  \emph{Some remarks on Hermitian manifolds satisfying K\"ahler-like conditions,} Math. Zeit. {\bf 298} (2021), 49-68.

\bibitem {FinoTomassini} A. Fino and A. Tomassini, \emph{A survey on strong KT structures,} Bull. Math. Soc. Sci. Math. Roumanie, Tome 52 (100) no. 2, 2009, 99-116.

\bibitem {FV}  A. Fino, and L. Vezzoni, \emph{On the existence of balanced and SKT metrics on nilmanifolds,} Proc. Amer. Math. Soc., {\bf 144} (2016), no.6, 2455-2459.

\bibitem {FV2}     A. Fino and L. Vezzoni, \emph{Special
Hermitian metrics on compact solvmanifolds,} J. Geom. Phys. {\bf 91} (2015), 40-53.

\bibitem {Fu} J-X Fu,  {\em On non-K\"ahler Calabi-Yau threefolds with balanced metrics.} Proceedings of the International
Congress of Mathematicians. Volume II, 705-716, Hindustan Book Agency, New Delhi, 2010.


\bibitem {Fu-Li-Yau} J-X Fu, J. Li, and S-T Yau, {\em Constructing balanced metrics on some families of non-K\"ahler Calabi-Yau threefolds,} J. Diff. Geom. {\bf 90} (2012), no. 1, 81-129.

\bibitem {FuWangWu} J-X Fu, Z. Wang, and D. Wu, {\em Semilinear equations, the $\gamma_k$ function, and generalized Gauduchon metrics.} J. Eur. Math. Soc.  {\bf 15} (2013), 659-680.



\bibitem {Fu-Yau} J-X Fu and S-T Yau, \emph{ The theory of superstring with flux on non-K\"ahler manifolds and the complex Monge-Amp\`ere equation,} J. Diff. Geom. {\bf 78} (2008), no.3, 369-428.




\bibitem {Fu-Zhou} J-X Fu and X. Zhou, \emph{Scalar curvatures in almost Hermitian geometry and some applications,} Sci. China Math. {\bf 65} (2022), 2583-2600. 

\bibitem  {GatesHR} S.J. Gates, C.M. Hull and M. Ro\u{c}ek, \emph{Twisted multiplets and new
supersymmetric nonlinear sigma models,} Nuc. Phys. B, {\bf 248} (1984), 157-186.

\bibitem {Gauduchon} P. Gauduchon, \emph{La $1$-forme de torsion d'une
vari\'et\'e hermitienne compacte.} Math. Ann. {\bf 267} (1984), no.4, 495-518.

\bibitem {Gauduchon1} P. Gauduchon, \emph{Hermitian connnections and Dirac operators,}
Boll. Un. Mat. It. {\bf 11-B} (1997) Suppl. Fasc., 257-288.

\bibitem  {Gray} A. Gray, \emph{Curvature identities for Hermitian and almost Hermitian manifolds,}  Tohoku Math. J. {\bf 28}
(1976), no. 4, 601-612.


\bibitem  {Hull} C. Hull, \emph{ Superstring compactifications with torsion and space-time supersymmetry.} In:
Turin 1985 Proceedings ``Superunification and Extra Dimensions" (1986), 347-375.



\bibitem  {IvanovP}  S. Ivanov and G. Papadopoulos, \emph{Vanishing theorems and string backgrounds,}
Classical Quantum Gravity {\bf 18} (2001), 1089-1110.

\bibitem {KYZ} G. Khan, B. Yang, and F. Zheng, \emph{The set of all orthogonal complex strutures on the flat $6$-torus,} Adv. Math. {\bf 319} (2017), 451-471.

\bibitem {Kodaira} K. Kodaira, \emph{On the structure of compact complex analytic spaces I,}  Am. J. Math. {\bf 86} (1964),
751-798; II, ibid. {\bf 88} (1966), 682-721; III, ibid. {\bf 90} (1969), 55-83; IV, ibid. {\bf 90} (1969), 1048-1066.

\bibitem {Li-Yau} J. Li and S.-T. Yau, \emph{The existence of supersymmetric string theory with torsion,} J. Differential Geom. \textbf{70}
(2005), no. 1, 143--181.

\bibitem {Liu-Yang}  K-F Liu and X-K Yang, \emph {Geometry of Hermitian manifolds,} Internat. J. Math. {\bf 23} (2012) (40 page)

\bibitem {Liu-Yang1} K-F Liu and X-K Yang, \emph {Ricci cuvratures on Hermitian manifolds,}  Trans. Amer. Math. Soc. {\bf 369} (2017), no. 7, 5157-5196.

\bibitem {Liu-Yang2} K-F Liu and X-K Yang, \emph {Hermitian harmonic maps and non-degenerate curvatures,}  Math. Res. Lett. {\bf 21} (2014), no. 4, 831-862.

\bibitem {OUV} A. Otal, L. Ugarte, R. Villacampa, \emph{Hermitian metrics on compact complex
manifolds and their deformation limits,} Special metrics and group actions in geometry, 269-
290, Springer INdAM Ser., {\bf 23}, Springer, Cham, 2017.

\bibitem {Popovici}   D. Popovici, \emph{Limits of projective manifolds under holomorphic deformations: Hodge numbers and strongly Gauduchon metrics,} Invent. Math. {\bf 194} (2013), no. 3, 515-534.

\bibitem {S} H. Samelson, \emph{A class of complex analytic manifolds, } Portugaliae Math. {\bf 12} (1953) 129-132.

\bibitem  {S18} J. Streets, \emph{Pluriclosed flow and the geometrization of complex surfaces,} arXiv:1808.09490, Geometric Analysis, Progress in Mathmematics vol. {\bf 333} (2020), pp. 471-510.

\bibitem  {ST10} J. Streets and G. Tian, \emph{A parabolic flow of pluriclosed metrics,} Int. Math. Res. Notices {\bf 16} (2010), 3101-3133.

\bibitem {ST13} J. Streets and G. Tian, \emph{Regularity results for pluriclosed flow,} Geometry and Topology, {\bf 17} (2013), 2389-2429.

\bibitem {Strominger} A. Strominger, \emph{Superstrings with Torsion,} Nuclear Phys. B {\bf 274} (1986), 253-284.

\bibitem {STW} G. Sz\'ekelyhidi, V. Tosatti and B. Weinkove, \emph {Gauduchon metrics with prescribed volume form,}  Acta Math. {\bf 219} (2017), no. 1, 181-211.

\bibitem {Tosatti} V. Tosatti,  \emph { Non-K\"ahler Calabi-Yau manifolds,} Analysis, complex geometry, and mathematical physics: in honor of Duong H. Phong, 261-277, Contemp. Math., 644, Amer. Math. Soc., Providence, RI, 2015. arXiv: 1401.4797.

\bibitem  {Tseng-Yau} L.-S. Tseng and S.-T. Yau, \emph{Non-K\"ahler Calabi-Yau
manifolds.} String-Math 2011, 241-254, Proc. Sympos. Pure Math.,
\textbf{85}, Amer. Math. Soc., Providence, RI, 2012.

\bibitem {VYZ} L. Vezzoni, B. Yang, and F. Zheng, \emph{ Lie groups with flat Gauduchon connections,}  Math. Zeit. {\bf 293} (2019), 597-608.

\bibitem {Wang} H.-C. Wang, \emph{Closed manifolds with homogeneous complex structure,} Amer. J. Math. {\bf 76}
(1954) 1-32.

\bibitem  {WYZ} Q. Wang, B. Yang, and F. Zheng, \emph{On Bismut flat manifolds,} Trans. Amer. Math. Soc. {\bf 373} (2020), 5747-5772.

\bibitem {YZ} B. Yang and F. Zheng, \emph{On curvature tensors of Hermitian manifolds,} Comm. Anal. Geom. {\bf 26} (2018), no.5, 1193-1220.


\bibitem {YZ1} B. Yang and F. Zheng, \emph{On compact Hermitian manifolds with flat Gauduchon conmnections,} Acta Math. Sinica (English Series). {\bf 34} (2018), 1259-1268.

\bibitem  {Yano} K. Yano. \emph{Differential geometry on complex and almost complex spaces.} International Series of Monographs in Pure and Applied Mathematics, Vol \textbf{49}, A Pergamon Press Book. 1965.

\bibitem  {YZZ} S.-T. Yau, Q. Zhao, and F. Zheng, \emph{On Strominger K\"ahler-like manifolds with degenerate torsion,} arXiv:1908.05322v2, to appear in Trans. Amer. Math. Soc.

\bibitem {ZZ1} Q. Zhao and F. Zheng, \emph{Complex nilmanifolds and K\"ahler-like connections,} J. Geom. Phys. {\bf 146} (2019).

\bibitem {Zheng} F. Zheng, \emph{Complex differential geometry.} AMS/IP Studies in Advanced Mathematics, 18. American
Mathematical Society, Providence, RI; International Press, Boston, MA, 2000.

\bibitem {Zheng1} F. Zheng, \emph{Some recent progress in non-K\"ahler geometry,} Sci. China Math., {\bf 62} (2019), no.11, 2423-2434.

\end{thebibliography}
\end{document}